\documentclass[english]{article}
\usepackage{amsmath}
\usepackage{amssymb}
\usepackage{amsthm}
\usepackage{enumerate}
\usepackage[all]{xy}

\newcommand{\sph}{\mathbb{S}}

\newcommand{\E}{\mathbb{E}}

\theoremstyle{definition} 
\newtheorem{rem}{Remark}

\theoremstyle{plain}
\newtheorem{thm}{Theorem}

\newtheorem{cor}{Corollary} 
\newtheorem{conj}{Conjecture}

\begin{document}

\title{Topological Hochschild homology of $X(n)$}
\author{Jonathan Beardsley}
\date{July 7, 2015}
\maketitle 

\begin{abstract}
We show that Ravenel's spectrum $X(2)$ is the versal $\E_1$-$\sph$-algebra of characteristic $\eta$. This implies that every $\E_1$-$\sph$-algebra $R$ of characteristic $\eta$ admits an $\E_1$-ring map $X(2)\to R$, i.e. an $\mathbb{A}_\infty$ complex orientation of degree 2. This implies that $R^\ast(\mathbb{C}P^2)\cong R_\ast[x]/x^3$. Additionally, if $R$ is an $\mathbb{E}_2$-ring \emph{Thom} spectrum admitting a map (of homotopy ring spectra) from $X(2)$, e.g. $X(n)$, its topological Hochschild homology has a simple description.  
\end{abstract}

\begin{thm}
The spectrum $X(n)$, which is the Thom spectrum of the inclusion $\Omega SU(n)\hookrightarrow \Omega SU\simeq BU\to BGL_1(\sph)$, is of characteristic $\eta$ in the sense of \cite{szymikprimechar} and \cite{acb}. In particular $X(2)$ is the versal $\E_1$-$\sph$-algebra of characteristic $\eta$ (as described in Definition 4.3 of \cite{acb}). 
\end{thm}

\begin{proof}
We use \cite{acb} in a crucial way. Recall that $X(2)$ is the Thom spectrum of the inclusion $i:\Omega S^3\simeq \Omega SU(2)\hookrightarrow BU\to BGL_1(\sph)$. Note that this morphism is a two fold loop map, and as such a morphism of $\E_2$-algebras in $Top$. Let $\tilde{i}$ be the induced $\E_1$-morphism.   We have the following equivalences of mapping spaces: $$Map_{\E_1-alg}(\Omega\Sigma S^2,BGL_1(\sph))\simeq Map_{Top}(S^2,BGL_1(\sph)\simeq Map_{Top}(S^1,GL_1(\sph)).$$ Since $\pi_1(GL_1(\sph))\cong \pi_1(\sph)\cong \mathbb{Z}/2$ we have that this map is either null homotopic or non-trivial and unique up to homotopy. Since it is not null (i.e.~the associated Thom spectrum is not the suspension spectrum of $\Omega S^3$), $\tilde{i}_\ast\colon \pi_\ast(S^1)\to \pi_\ast(GL_1(\sph))$ takes $1\in\pi_1(S^1)$ to $\eta$, the generator of $\pi_1(\sph)\cong \pi_1(GL_1(\sph))$. Indeed, the preceding sequence of equivalences implies, by Theorem 4.10 of \cite{acb}, that $X(2)\simeq \sph/\!\!/_{\E_1}\eta$, the versal $\E_1$-algebra over $\sph$ of characteristic $\eta$.  

Moreover, as $X(n)$ admits a morphism of $\E_1$-ring spectra (actually of $\E_2$-ring spectra) $X(2)\to X(n)$ for every $n$, we have that the $X(n)$ must be an $\E_1$-$\sph$-algebra of characteristic $\eta$. In particular, the composition $\Sigma\sph\overset{\eta}\to\sph\to X(n)$ is nullhomotopic. 
\end{proof}

\begin{rem}
Of course it's not necessary to use the machinery of characteristics of structured ring spectra to notice that $\eta$ is trivial in $X(n)_\ast$, but the identification of $X(2)$ as the versal $\E_1$-$\sph$-algebra of characteristic $\eta$ seems interesting in its own right. 
\end{rem}

\begin{cor}
If $R$ is an $\E_1$-ring spectrum of characteristic $\eta$ then $Map_{\E_1}(X(2),R)\simeq \Omega^{\infty+2} R$ and $R^\ast(\mathbb{C}P^2)\simeq R_\ast[x]/x^3$. 
\end{cor}

\begin{proof}
The first statement follows immediately from Lemma 4.6 of \cite{acb} and the ``versality" of $X(2)$. Hence there is at least one $\E_1$-morphism from $X(2)$ to $R$. From Proposition 6.5.4 of \cite{rav} we obtain the second part of the corollary. 
\end{proof}

\begin{thm}
The topological Hochschild homology of $X(n)$ as an $\E_2$-ring spectrum, denoted here by $THH(X(n))$, is equivalent to $X(n)\wedge SU(n)_+$.  
\end{thm}

\begin{proof}
Here we use \cite{bcs} is a crucial way. In particular, we recall Theorem 2 of that paper which gives $THH(X(n))=X(n)\wedge M(\eta\circ Bi)$, where $\eta\circ Bi$ here refers to the morphism $$B\Omega SU(n)\simeq SU(n)\overset{Bi}\to B^2GL_1(\sph)\overset{\eta}\to BGL_1(\sph).$$ Since $X(n)$ is of characteristic $\eta$, we have that the composition $B^2GL_1(\sph)\overset{\eta}\to BGL_1(\sph)\to BGL_1(X(n))$ is nullhomotopic, where $BGL_1(\sph)\to BGL_1(X(n))$ is just $BGL_1(-)$ of the unit map of $X(n)$. This implies that $M(\eta\circ Bi)$ is $X(n)$-oriented, so by the associated Thom isomorphism we have $X(n)\wedge M(\eta\circ Bi)\simeq X(n)\wedge SU(n)_+$. 
\end{proof}

\begin{rem}
By a similar argument, given any Thom spectrum $Mf$, for $f\colon X\to BGL_1(\sph)$ a map of double loop spaces, such that the unit map $\sph\to Mf$ factors (as maps of $\E_2$-rings) $\sph\to X(2)\to Mf$, we have that $THH(Mf)\simeq Mf\wedge \Omega X_+$. 
\end{rem}

\begin{conj}
Recall that the morphism of $\E_2$-ring spectra $X(n)\to X(n+1)$ is a Hopf-Galois extension with associated spectral Hopf-algebra $\Omega S^{2n+1}$, thought of as the base space of the fibration $\Omega SU(n)\to \Omega SU(n+1)\to \Omega S^{2n+1}$. Then the above results, as well as the results of \cite{bcs} suggest that one might have relative $THH$ spectra: $$THH_{X(n)}(X(n+1))\simeq X(n+1)\wedge S^{2n+1}_+.$$ 
\end{conj}

\bibliographystyle{amsalpha}
\bibliography{references}

\end{document}